\documentclass[12pt]{amsart}
\usepackage{latexsym, amssymb, amsmath}

\newtheorem{theorem}{Theorem}[section]
\newtheorem{lemma}[theorem]{Lemma}

\newtheorem{proposition}[theorem]{Proposition}
\newtheorem{remark}[theorem]{Remark}
\theoremstyle{definition}

\newtheorem{problem}{Problem}

\makeatletter
    
    \@addtoreset{equation}{section}
  \makeatother

\newcommand{\Ric}{{\rm Ric}}

\newcommand{\n}{\nabla}
\begin{document}

\title[3-dimensional gradient Yamabe solitons]
{3-dimensional complete gradient Yamabe solitons with divergence-free Cotton tensor}


%    author one information

\author{Shun Maeta}
\address{Department of Mathematics,
 Shimane University, Nishikawatsu 1060 Matsue, 690-8504, Japan.}
\curraddr{}
\email{shun.maeta@gmail.com~{\em or}~maeta@riko.shimane-u.ac.jp}
\thanks{The author is partially supported by the Grant-in-Aid for Young Scientists,
 No.15K17542 and No.19K14534 Japan Society for the Promotion of Science, and JSPS Overseas Research Fellowships 2017-2019 No.70.}

%    \subjclass is required.

\date{}

\dedicatory{}


%    "Communicated by" -- provide editor's name; required.
\commby{}

%    Abstract is required.
\begin{abstract}
In this paper, we classify 3-dimensional complete gradient Yamabe solitons with divergence-free Cotton tensor. We also give some classifications of complete gradient Yamabe solitons with nonpositively curved Ricci curvature in the direction of the gradient of the potential function. 
\end{abstract}

\maketitle


\bibliographystyle{amsplain}

\section{Introduction}\label{intro} 
 A Riemannian manifold $(M^n,g)$ is called a gradient Yamabe soliton if there exist a smooth function $F$ on $M$ and a constant $\rho\in \mathbb{R}$, such that 
\begin{equation}\label{YS}
(R-\rho)g=\nabla\nabla F,
\end{equation}
where $R$ is the scalar curvature on $M$ and $\n\n F$  is the Hessian of $F$.
If $\rho>0$, $\rho=0$, or $\rho<0$, then the Yamabe soliton is called shrinking, steady, or expanding.
%By scaling the metric, we can assume $\rho=1,0,-1$, respectively. 
If the potential function $F$ is constant, then the Yamabe soliton is called trivial. It is known that any compact Yamabe soliton is trivial (see for example \cite{CMM12},~\cite{Hsu12}). Yamabe solitons are special solutions of the Yamabe flow which was introduced by R. Hamilton \cite{Hamilton89}.
The Yamabe soliton equation $(\ref{YS})$ is similar to the equation of Ricci solitons.
%$$\Ric+\n \n f=\rho g,$$
%where $\Ric$ is the Ricci tensor of $M$. 
Ricci solitons are special solutions of the Ricci flow 
%$\frac{\partial}{\partial t}g=-2\Ric$ 
which was also introduced by R. Hamilton \cite{Hamilton82}.
As is well known, by using the Ricci flow, G. Perelman \cite{Perelman1},~\cite{Perelman2},~\cite{Perelman3} proved Thurston's geometrization conjecture \cite{Thurston} and Poincar\'e conjecture.
In the first paper of Perelman, he mentioned that ``any 3-dimensional complete noncompact $\kappa$-noncollapsed gradient steady Ricci soliton with positive curvature  is rotationally symmetric, namely Bryant soliton".
In \cite{CCCMM14}, H.-D. Cao, G. Catino, Q. Chen, C. Mantegazza and L. Mazzieri gave a partial answer to the conjecture.
Finally, S. Brendle proved the conjecture \cite{Brendle}. 
In this paper, we consider the similar problem.
More precisely, we consider the following problem.
\begin{problem}\label{PCY}
Classify nontrivial non-flat complete 3-dimensional gradient Yamabe solitons.
\end{problem}
P. Daskalopoulos and N. Sesum \cite{DS13} showed that ``all locally conformally flat complete gradient Yamabe solitons with positive sectional curvature have to be rotationally symmetric". The proof was inspired by H.-D. Cao and Q. Chen's paper \cite{CC11}.
Furthermore, they constructed some examples of rotationally symmetric gradient Yamabe solitons on $\mathbb{R}^n$ with positive sectional curvature.
Recently,  H.-D. Cao, X. Sun and Y. Zhang relaxed the assumption, and showed that any nontrivial non-flat complete and locally conformally flat gradient Yamabe soliton with nonnegative scalar curvature is rotationally symmetric. 
G. Catino, C. Mantegazza and L. Mazzieri's work \cite{CMM12} is also important.
They classified complete conformal gradient solitons with nonnegative Ricci tensor. As a corollary, they classified nontrivial complete gradient Yamabe solitons with nonnegative Ricci tensor. Finally, it is shown that complete gradient Yamabe solitons are rotationally symmetric under (1) nonnegative Ricci tensor is positive definite at some point, by Catino, Mantegazza and Mazzieri \cite{CMM12}, or (2) positive Ricci curvature, by Cao, Sun and Zhang \cite{CSZ12}.
Therefore, in this paper, we consider $3$-dimensional Yamabe solitons without any assumptions for non-negativity of curvatures.
Our main theorem gives an affirmative partial answer to Problem~\ref{PCY}:

\begin{theorem}\label{main}
Let $(M^3,g,F)$ be a nontrivial non-flat $3$-dimensional complete gradient Yamabe soliton with divergence-free Cotton tensor $($i.e., Bach flat$)$. 

${\rm I}.$ If $M$ is steady, then $M$ is rotationally symmetric and equal to the warped product 
$$([0,\infty),dr^2)\times_{|\n F|}(\mathbb{S}^{2},{\bar g}_{S}),$$
where $\bar g_{S}$ is the round metric on $\mathbb{S}^{2}.$

${\rm II}$. If $M$ is shrinking, then either

$(1)$ $M$ is rotationally symmetric and equal to the warped product 
$$([0,\infty),dr^2)\times_{|\n F|}(\mathbb{S}^{2},{\bar g}_{S}),$$
where $\bar g_{S}$ is the round metric on $\mathbb{S}^{2},$ or

$(2)$ $|\n F|$ is constant and $M$ is isometric to the Riemannian product 
$$(\mathbb{R},dr^2)\times \left(\mathbb{S}^2\left(\frac{1}{2}\rho|\n F|^2\right),\bar g\right),$$
where $\mathbb{S}^2(\frac{1}{2}\rho|\n F|^2)$ is the sphere of constant Gaussian curvature $\frac{1}{2}\rho|\n F|^2$.

${\rm III}$. If $M$ is expanding, then either

$(1)$ $M$ is rotationally symmetric and equal to the warped product 
$$([0,\infty),dr^2)\times_{|\n F|}(\mathbb{S}^{2},{\bar g}_{S}),$$
where $\bar g_{S}$ is the round metric on $\mathbb{S}^{2},$ or

$(2)$ $|\n F|$ is constant and $M$ is isometric to the Riemannian product 
$$(\mathbb{R},dr^2)\times \left(\mathbb{H}^2\left(\frac{1}{2}\rho|\n F|^2\right),\bar g\right),$$
where $\mathbb{H}^2(\frac{1}{2}\rho|\n F|^2)$ is the hyperbolic space of constant Gaussian curvature $\frac{1}{2}\rho|\n F|^2$.

\end{theorem}

\begin{remark}
For dimension $n\geq4$, Bach \cite{Bach} introduced the Bach tensor in $1920$'s.
\begin{align}
B_{ij}
=&\frac{1}{n-3}\n^k\n^lW_{ikjl}+\frac{1}{n-2}R_{kl}W_i{}^{k}{}_j{}^l\\
=&\frac{1}{n-2}(\n_kC_{kij}+R_{kl}W_i{}^{k}{}_j{}^l),\notag
\end{align}
where $\nabla$ is the Levi-Civita connection, $W$ is the Weyl tensor, $R_{ij}$ is the Ricci tensor, and $C$ is the Cotton tensor.
 In \cite{CCCMM14}, the Bach tensor for $3$-dimensional manifolds was introduced as follows:
$$B_{ij}=\n_kC_{kij}.$$

%$(2)$ On a $3$-dimensional manifold $M$, $M$ is locally conformally flat if and only if the Cotton tensor vanishes. 

%$(3)$ We also remark that by the equations of Yamabe solitons and Ricci solitons, Yamabe solitons contain less information than Ricci solitons.
\end{remark}

The remaining sections are organized as follows. Section~$\ref{Pre}$ contains some necessary definitions and preliminary geometric results.
%In section~$\ref{formulasY}$, we show some formulas of Yamabe solitons. 
Section~$\ref{Proof of main}$ is devoted to the proof of Theorem~$\ref{main}$. 
In section~\ref{NRIC}, we consider complete gradient Yamabe solitons with nonpositively curved Ricci curvature in the direction of the gradient of the potential function.

%%%%%%%%%%%%%%%%%%%%%%%%%%%%%%%%%%%

\section{Preliminary}\label{Pre}
The Riemannian curvature tensor is defined by 
$$R(X,Y)Z=-\n_X\n_YZ+\n_Y\n_XZ+\n_{[X,Y]}Z.$$
%where $\nabla$ is the Levi-Civita connection.
%Here we remark that the tensor is not the same notation of Hamilton's paper \cite{Hamilton82}.
The Ricci tensor $R_{ij}$ (also denoted by $\Ric$) is defined by 
$R_{ij}=R_{ipjp}.$
The Weyl tensor $W$ and  the Cotton tensor C are defined by 
\begin{align*}
R_{ijkl}
=&W_{ijkl}
+\frac{R}{(n-1)(n-2)}(g_{il}g_{jk}-g_{ik}g_{jl})\\
&-\frac{1}{n-2}(R_{il}g_{jk}+R_{jk}g_{il}-R_{ik}g_{jl}-R_{jl}g_{ik})\\
=&W_{ijkl}+S_{ik}g_{jl}+S_{jl}g_{ik}-S_{il}g_{jk}-S_{jk}g_{il},
\end{align*}
\begin{align*}%\label{CT}
C_{ijk}
=&\nabla_iR_{jk}-\nabla_jR_{ik}-\frac{1}{2(n-1)}(g_{jk}\nabla_iR-g_{ik}\nabla_jR)\\
=&\n_iS_{jk}-\n_jS_{ik},\notag
\end{align*}
where $S=\Ric-\frac{1}{2(n-1)}Rg$ is the Schouten tensor.
 The Cotton tensor is skew-symmetric in the first two indices and totally trace free, that is,
$$C_{ijk}=-C_{jik} \quad \text{and} \quad g^{ij}C_{ijk}=g^{ik}C_{ijk}=0.$$
%The Weyl tensor and the Cotton tensor are also useful tensors for Ricci solitons.
As is well known, a Riemannian manifold $(M^n,g)$ is locally conformally flat if and only if 
(1) for $n\geq4$, the Weyl tensor vanishes; (2) for $n=3$, the Cotton tensor vanishes.
Moreover, for $n\geq4$, if the Weyl tensor vanishes, then the Cotton tensor vanishes. We also see that for $n=3$, the Weyl tensor always vanishes, but the Cotton tensor does not vanish in general.

We prove some formulas needed later. 
Taking trace of the Yamabe soliton equation \eqref{YS},
\begin{equation}\label{TYS}
n(R-\rho)=\Delta F,
\end{equation}
where $\Delta$ is the Laplacian on $M$. 
In general, we have
\begin{equation}\label{p.1}
\Delta {\nabla}_iF={\nabla}_i\Delta F+R_{ij}{\nabla}_jF.
\end{equation}
Substituting 
\begin{equation*}
\Delta {\nabla}_iF={\nabla}_k{\nabla}_k{\nabla}_iF={\nabla}_k((R-\rho)g_{ki})={\nabla}_iR,
\end{equation*}
and
\begin{equation*}
{\nabla}_i\Delta F={\nabla}_i(n(R-\rho))=n{\nabla}_iR,
\end{equation*}
into $(\ref{p.1})$, we have
\begin{equation}\label{p.2}
(n-1)\nabla_iR+R_{il}\nabla_lF=0.
\end{equation}
Thus, we have 
\begin{equation}\label{p.3}
(n-1)g(\nabla R,\nabla F)=-\Ric(\nabla F,\nabla F).
\end{equation}
On the other hand, by $(\ref{p.2})$ and the contracted second Bianchi identity, 
\begin{equation}\label{p.4}
(n-1)\Delta R+\frac{1}{2}g(\nabla R, \nabla F)+R(R-\rho)=0.
\end{equation}
Combining $(\ref{p.3})$ with $(\ref{p.4})$, we obtain
\begin{equation}\label{p.5}
\Delta R=\frac{1}{2(n-1)^2}\Ric(\nabla F,\nabla F)-\frac{1}{n-1}R(R-\rho).
\end{equation}
\quad\\

%%%%%%%%%%%%%%%%%%%%%%%%%%%%%%
%%%%%%%%%%%%%%%%%%%%%%%%%%%%%%%%%%%%%%

\section{Proof of Theorem $\ref{main}$}\label{Proof of main}
In this section, we prove Theorem $\ref{main}$.
%To prove Theorem \ref{main}, by using lemmas of Section \ref{formulasY}, we give the following lemma.
To prove Theorem \ref{main}, we use the following useful theorem by H.-D. Cao, X. Sun and Y. Zhang:

\begin{theorem}[\cite{CSZ12}]\label{Thm of CSZ12}
Let $(M^n,g,F)$ be a nontrivial complete gradient Yamabe soliton. Then, $|\n F|^2$ is constant on regular level surfaces of $F$, and either

$(1)$ $F$ has a unique critical point at some point $p_0\in M$, and $M$ is rotationally symmetric and equal to the warped product 
$$([0,\infty),dr^2)\times_{|\n F|}(\mathbb{S}^{n-1},{\bar g}_{S}),$$
where $\bar g_{S}$ is the round metric on $\mathbb{S}^{n-1},$ or

$(2)$ $F$ has no critical point and $M$ is the warped product 
$$(\mathbb{R},dr^2)\times_{|\n F|}(N^{n-1},\bar g),$$
where N is a Riemannian manifold of constant scalar curvature.
Furthermore, if the Ricci curvature of $N$ is nonnegative, then $M$ is isometric to the Riemannian product 
$(\mathbb{R},dr^2)\times(N^{n-1},\bar g)$; if $R\geq0$, then either $R>0$, or $R=\overline R=0$ and $(M,g)$ is isometric to the Riemannian product $(\mathbb{R},dr^2)\times(N^{n-1},\bar g)$.
\end{theorem}

\begin{proof}[Proof of Theorem~$\ref{main}$]

We only have to consider the case (2) of Theorem $\ref{Thm of CSZ12}$.
Since 
\begin{align*}
\n_iB_{ij}
=&\n_i\n_kC_{kij}\\
=&\n_i\n_k(\n_kS_{ij}-\n_iS_{kj})\\
=&\n_i\n_k\n_kS_{ij}-\n_k\n_i\n_kS_{ij}\\
=&R_{ikkp}\n_pS_{ij}+R_{ikip}\n_kS_{pj}+R_{ikjp}\n_kS_{ip}\\
=&-R_{ip}\n_pS_{ij}+R_{kp}\n_kS_{pj}+(S_{ij}g_{kp}+S_{kp}g_{ij}-S_{ip}g_{kj}-S_{kj}g_{ip})\n_kS_{ip}\\
=&S_{ij}C_{kik}+S_{ip}C_{ijp}\\
=&-C_{jip}R_{ip},
\end{align*}
\begin{align}\label{m1}
\n_i\n_jB_{ji}=-\n_iC_{ijk}R_{jk}-C_{ijk}\n_iR_{jk}.
\end{align}
By the definition and a property of the Cotton tensor,
\begin{align*}
C_{ijk}\n_iR_{jk}
=&C_{ijk}(C_{ijk}+\n_jR_{ik}+\frac{1}{4}(g_{jk}\n_iR-g_{ik}\n_jR))\\
=&|C_{ijk}|^2-C_{jik}\n_jR_{ik}.
\end{align*}
Thus, we have
\begin{equation}\label{m2}
C_{ijk}\n_iR_{jk}=\frac{1}{2}|C_{ijk}|^2.
\end{equation}
Substituting (\ref{m2}) into (\ref{m1}), we have
$$\n_i\n_jB_{ji}=-B_{jk}R_{jk}-\frac{1}{2}|C_{ijk}|^2.$$
By the assumption, the Cotton tensor vanishes.

As in the proof of  Theorem \ref{Thm of CSZ12}, it is shown that 
%(a) $|\n F|^2$ and $R$ is constant on regular level surface $N^{n-1}=F^{-1}(c)$ of $F$, and 
%(b) 
in any open neighborhood $U$ of $N^{2}$ in which $F$ has no critical points, 
$$g=dr^2+(F'(r))^2{\bar g}=dr^2+\frac{(F'(r))^2}{(F'(r_0))^2}g_{ab}(r_0,x)dx^adx^b,$$ where $(x^2, x^3)$ is any local coordinates system on $N^{2}$ and $\bar g=(F'(r_0))^{-2}\bar g_{r_0}$, where $\bar g_{r_0}$ is the induced metric on $N^{2}$.%=r^{-1}(r_0).$

By a direct calculation, we can get formulas of the warped product manifold of the warping function $|\n F|=F'(r).$
 For $a,b,c,d=2,3,$
\begin{align}\label{RT1}
R_{1a1b}&=-F'F'''{\bar g}_{ab},\quad R_{1abc}=0,\\
R_{abcd}&=(F')^2{\bar R}_{abcd}+(F'F'')^2(\bar g_{ad}\bar g_{bc}-\bar g_{ac}\bar g_{bd}),\notag
\end{align}
\begin{align}\label{RT2}
R_{11}=&-2\frac{F'''}{F'},\quad 
R_{1a}=0,\\
R_{ab}=&\bar R_{ab}-((F'')^2+F'F''')\bar g_{ab},\notag
\end{align}
\begin{align}\label{RT3}
R=(F')^{-2}\bar R-2\Big(\frac{F''}{F'}\Big)^2-4\frac{F'''}{F'},
\end{align}
where the curvature tensors with bar are the curvature tensors of $(N,\bar g)$.
By (\ref{YS}), 
\begin{align}\label{R-rho}
R-\rho=F''.
\end{align}
Since $(N^2,{\overline g})$ is a 2-dimensional manifold,
$${\overline R}_{abcd}=-\frac{\bar R}{2}(\bar g_{ad} \bar g_{bc}-\bar g_{ac}\bar g_{bd}),$$
$$\bar R_{ad}=\frac{\bar R}{2}\bar g_{ad}.$$
Substituting these into (\ref{RT1}) and (\ref{RT2}),  we have
\begin{align}\label{RT11}
R_{1a1b}&=-F'F'''{\bar g}_{ab},\quad R_{1abc}=0,\\
R_{abcd}&=-(F')^3\Big(\frac{1}{2}F'R+2F'''\Big)(\bar g_{ad}\bar g_{bc}-\bar g_{ac}\bar g_{bd}),\notag
\end{align}
\begin{align}\label{RT22}
R_{11}=&-2\frac{F'''}{F'},\quad 
R_{1a}=0,\\
R_{ab}=&\Big(\frac{R}{2}(F')^2+F'F'''\Big)\bar g_{ab}.\notag
\end{align}
%\begin{align}\label{RT33}
%R=(F')^{-2}\bar R-2\Big(\frac{F''}{F'}\Big)^2-4\frac{F'''}{F'}.
%\end{align}
Hence, the Cotton tensor $C_{ijk}$ is 
\begin{equation*}
C_{ijk}=\left\{
\begin{aligned}
&\nabla_1(R_{ab}-\frac{1}{4}Rg_{ab})
&\quad (a,b=2,3),\\
&0\qquad (other).
\end{aligned}
\right.
\end{equation*}
Thus, we have 
\begin{equation}\label{key2}
\frac{R}{4}(F')^2+F'F'''=c~(\text{constant}).
\end{equation}
Combining \eqref{key2} with $(\ref{RT3})$, we have
$$(F'')^2=\frac{1}{2}\bar R -2c.$$
Since $\bar R$ is constant, $F''$ is constant.
Thus, $R$ is constant by the Yamabe soliton equation.
Therefore the equation $(\ref{key2})$ is as follows. 
\begin{equation}\label{key3}
\frac{1}{4}R(F')^2=c.
\end{equation}

If $c=0,$ then $R=0$. From this and $(\ref{RT11})$, $M$ is flat. 

If $c\not=0$, then we have $R\not=0$ and $F'$ is constant. Thus, $R-\rho=F''=0.$

Case I. $M$ is steady: We have $R=\rho=0$, which is a contradiction.

Case II. $M$ is shrinking: Since $R=\rho>0$ and $(\ref{RT3})$,
$\bar R=R(F')^2=\rho|\n F|^2>0$.

Case III. $M$ is expanding: Since $R=\rho<0$ and $(\ref{RT3})$,
$\bar R=R(F')^2=\rho|\n F|^2<0$.

\end{proof}

\section{Complete gradient Yamabe solitons with $\Ric(\nabla F,\nabla F)\leq0$}\label{NRIC}

As mentioned before, H.-D. Cao, X. Sun and Y. Zhang showed that any nontrivial non-flat complete and locally conformally flat gradient Yamabe soliton with $R\geq0$ is rotationally symmetric. Therefore, in this section, we consider Yamabe solitons with $\Ric(\n F,\n F)\leq0$ instead of ``locally conformally flat".

\begin{proposition}\label{prop1}
Let $(M^n,g,F)$ be an $n$-dimensional complete gradient Yamabe soliton with $\Ric(\n F, \n F)\leq0$. 
Suppose that $F$ has no critical point. Then, the following holds.

$(1)$ $M$ is shrinking or steady: If $R\geq\rho$, then $R=\rho.$

$(2)$ There exists no expanding soliton with $R\geq0.$
\end{proposition}

As a corollary, by the similar argument as in the proof of Theorem $\ref{main}$, we can classify nontrivial non-flat complete 3-dimensional gradient Yamabe solitons:

If $M$ is shrinking with $R\geq\rho$ and $\Ric(\n F,\n F)\leq0$, then either

(1) $M$ is rotationally symmetric and equal to the warped product 
$$([0,\infty),dr^2)\times_{|\n F|}(\mathbb{S}^{2},{\bar g}_{S}),$$
where $\bar g_{S}$ is the round metric on $\mathbb{S}^{2}$, or

(2) $|\n F|$ is constant and $M$ is isometric to the Riemannian product 
$$(\mathbb{R},dr^2)\times \left(\mathbb{S}^2\left(\frac{1}{2}\rho|\n F|^2\right),\bar g\right),$$
where $\mathbb{S}^2(\frac{1}{2}\rho|\n F|^2)$ is the sphere of constant Gaussian curvature $\frac{1}{2}\rho|\n F|^2$.

If $M$ is steady or expanding with $R\geq0$ and $\Ric(\n F,\n F)\leq0$, then 
$M$ is rotationally symmetric and equal to the warped product 
$$([0,\infty),dr^2)\times_{|\n F|}(\mathbb{S}^{2},{\bar g}_{S}).$$
\\

\begin{proof}[Proof of Proposition $\ref{prop1}$]
As in the proof of Theorem $\ref{Thm of CSZ12}$, it is shown that
 in any open neighborhood $U$ of $N^{n-1}$ in which $F$ has no critical points, 
$$g=dr^2+(F'(r))^2{\bar g}=dr^2+\frac{(F'(r))^2}{(F'(r_0))^2}g_{ab}(r_0,x)dx^adx^b,$$ where $(x^2,\cdots, x^n)$ is any local coordinates system on $N^{n-1}$ and $\bar g=(F'(r_0))^{-2}\bar g_{r_0}$, where $\bar g_{r_0}$ is the induced metric on $N^{n-1}$.%=r^{-1}(r_0).$

By a direct calculation, we can get formulas of the warped product manifold of the warping function $|\n F|=F'(r).$
 For $a,b,c,d=2,\cdots,n,$
\begin{align}\label{RT1-2}
R_{1a1b}&=-F'F'''{\bar g}_{ab},\quad R_{1abc}=0,\\
R_{abcd}&=(F')^2{\bar R}_{abcd}+(F'F'')^2(\bar g_{ad}\bar g_{bc}-\bar g_{ac}\bar g_{bd}),\notag
\end{align}
\begin{align}\label{RT2-2}
R_{11}=&-(n-1)\frac{F'''}{F'},\quad 
R_{1a}=0,\\
R_{ab}=&\bar R_{ab}-((n-2)(F'')^2+F'F''')\bar g_{ab},\notag
\end{align}
\begin{align}\label{RT3-2}
R=(F')^{-2}\bar R-(n-1)(n-2)\Big(\frac{F''}{F'}\Big)^2-2(n-1)\frac{F'''}{F'}.
\end{align}
By (\ref{YS}), 
\begin{align}\label{R-rho-2}
R-\rho=F''.
\end{align}
Since $\n F=F'\frac{\partial}{\partial r}$, 
\begin{equation}\label{Ricnf}
\Ric(\n F,\n F)=(F')^2R_{11}=-(n-1)F'F'''.
\end{equation}
By the assumption, $R'=F'''\geq0$.
By the definition of the Laplacian, 
\begin{align*}
\Delta R
=&g^{ij}(\partial_i\partial_jR-\Gamma_{ij}^k\partial_kR)\\
=&R''-g^{ij}\Gamma_{ij}^1R',
\end{align*}
where $\partial_1=\frac{\partial}{\partial r}$ and $\partial _i=\frac{\partial}{\partial x_i},~(i=2,\cdots,n).$
The Christoffel symbol is given by
\begin{align*}
\Gamma_{ij}^1
=&\frac{1}{2}g^{1k}(\partial_ig_{jk}+\partial_{j}g_{ik}-\partial_kg_{ij})\\
=&\frac{1}{2}(\partial_ig_{j1}+\partial_{j}g_{i1}-\partial_1g_{ij}).
\end{align*}
Here, 
$$\partial_1g_{11}=0, \quad \partial_ag_{11}=0\quad\text{and}\quad \partial_1g_{ab}=2F'F''\bar g_{ab}.$$
Thus, we have
$$\Gamma_{11}^1=0,\quad \Gamma_{1a}^1=0\quad\text{and}\quad\Gamma_{ab}^1=-F'F''\bar g_{ab}.$$
Hence,
\begin{equation}\label{DR}
\Delta R=R''+(n-1)\frac{F''}{F'}R'.
\end{equation}
Combining $(\ref{DR})$ with $(\ref{p.5})$,
\begin{equation}\label{R''}
R''=-(n-1)\frac{F''}{F'}R'-\frac{1}{2(n-1)}F'R'-\frac{1}{n-1}R(R-\rho).
\end{equation}

Case $(1)$, $M$ is shrinking or steady:
 By the assumption, $R\geq\rho(\geq0)$, that is, $F''\geq0$.
From this and $R'\geq0$, 
$$F^{(4)}=R''\leq0.$$
Thus, $F''$ is a non-negative weakly concave function, which means that $F''$ must be constant.
Hence, $R$ is constant. By $(\ref{R''})$, $R=\rho$.

Case $(2)$, $M$ is expanding: Assume that there exists an expanding Yamabe soliton with $R\geq0.$ The same argument shows $R=0,$ that is, $F''=-\rho<0$.
From this, $F'(>0)$ is a non-constant linear function, which cannot happen.

\end{proof}

\begin{remark}
If we assume that $\Ric (\n F,\n F)\geq0$ instead of $\Ric(\n F,\n F)\leq0$ on Proposition $\ref{prop1}$, then we immediately obtain $R=\rho$ without the assumption $R\geq\rho$ $($or $R\geq0$$)$.
In fact, by $(\ref{Ricnf})$, $F'''\leq0.$ Thus, $F'$ is a positive weakly concave function, which means that $F'$ must be constant.
Therefore, $R=\rho$.

As a result, we can get the same classification as in Theorem $\ref{main}$ for $n=3$, under $\Ric(\n F, \n F)\geq0$ instead of flatness of the Bach tensor.
\end{remark}

By the similar argument as in the proof of Theorem $\ref{main}$, 
we can get the following classification of complete gradient Yamabe solitons with $\Ric(\nabla F, \nabla F)\leq0$.

\begin{lemma}\label{lem1}
Let $(M^n,g,F)$ be an $n$-dimensional complete gradient Yamabe soliton with Ricci curvature bounded from below and $\Ric(\n F, \n F)\leq0$. Suppose that $F$ has no critical point.
Then, the following holds.

$(1)$ There exists no shrinking soliton with $R\leq0$.

$(2)$ $M$ is steady or expanding: If $R\leq \rho$, then $R=\rho.$

\end{lemma}

\begin{proposition}
Let $(M^3,g,F)$ be a nontrivial non-flat $3$-dimensional complete gradient Yamabe soliton with Ricci curvature bounded from below and $\Ric(\n F, \n F)\leq0$.

${\rm I}.$ $M$ is shrinking or steady: If $R\leq0$, then $M$ is rotationally symmetric and equal to the warped product 
$$([0,\infty),dr^2)\times_{|\n F|}(\mathbb{S}^{2},{\bar g}_{S}),$$
where $\bar g_{S}$ is the round metric on $\mathbb{S}^{2}.$

${\rm II}$. $M$ is expanding: If $R\leq\rho$, then either

$(1)$ $M$ is rotationally symmetric and equal to the warped product 
$$([0,\infty),dr^2)\times_{|\n F|}(\mathbb{S}^{2},{\bar g}_{S}),~\text{or}$$

$(2)$ $|\n F|$ is constant and $M$ is isometric to the Riemannian product 
$$(\mathbb{R},dr^2)\times \left(\mathbb{H}^2\left(\frac{1}{2}\rho|\n F|^2\right),\bar g\right),$$
where $\mathbb{H}^2(\frac{1}{2}\rho|\n F|^2)$ is the hyperbolic space of constant Gaussian curvature $\frac{1}{2}\rho|\n F|^2$.
\end{proposition}

\begin{proof}[Proof of Lemma $\ref{lem1}$]
$(1)$ $M$ is shrinking: Assume that there exists a Yamabe soliton with $R\leq0$. Set $L=-R\geq0$.
By $(\ref{p.5})$, 
\begin{align*}
\Delta L
=&-\frac{1}{2(n-1)^2}\Ric (\n F,\n F)+\frac{1}{n-1}L(L+\rho)\\
\geq&-\frac{1}{2(n-1)^2}\Ric (\n F,\n F)+\frac{1}{n-1}L^2.
\end{align*} 
By the assumption $\Ric(\n F,\n F)\leq0$, we obtain
$$\Delta L\geq\frac{1}{n-1}L^2.$$
Since $L$ is nonnegative, by Omori-Yau's maximum principle, $L=0.$ Thus, $F''=-\rho$ (as in the proof of Proposition~\ref{prop1}). Hence, $F'(>0)$ is a non-constant linear function, which cannot happen.

$(2)$ $M$ is steady or expanding:
 Set $u=\rho-R$. By $(\ref{p.5})$, 
\begin{align*}\Delta u
=&-\frac{1}{2(n-1)^2}\Ric (\n F,\n F)+\frac{1}{n-1}R(R-\rho)\\
\geq&-\frac{1}{2(n-1)^2}\Ric (\n F,\n F)+\frac{1}{n-1}u^2.
\end{align*} 
By the assumption $\Ric(\n F,\n F)\leq0$, we obtain
$$\Delta u\geq\frac{1}{n-1}u^2.$$
Since $u$ is nonnegative, by Omori-Yau's maximum principle, $u=0,$ that is, $R=\rho.$

\end{proof}

\noindent
{\bf Acknowledgments.}~
The work was done while the author was visiting the Department of Mathematics of Texas A $\&$ M University-Commerce as a Visiting Scholar and he is grateful to the department and the university for the hospitality he had received during the visit.

%%%%%%%%%%%%%%%%%%%%%%%%%%%%%%
%%%%%%%%%%%%%%%%%%%%%%%%%%%%%%
%%%%%%%%%%%%%%%%%%%%%%%%%%%%%%

%%%%%%%%%%%%%%%%%%%%%%%%%%%%
%%%%%%%%%%%%%%%%%%%%%%%%%%%%
%%%%%%%%%%%%%%%%%%%%%%%%%%%%

%%%%%%%%%%%%%%%%%%%%%%%%%%%%%%%%%%%%%%
\bibliographystyle{amsbook}

\end{document}